\def\year{2022}\relax
\documentclass[letterpaper]{article} 
\usepackage{aaai22} 
\usepackage{times} 
\usepackage{helvet} 
\usepackage{courier} 
\usepackage[hyphens]{url} 
\usepackage{graphicx, subfigure} 
\usepackage{graphicx} 
\usepackage{natbib} 
\usepackage{caption} 
\DeclareCaptionStyle{ruled}%
{labelfont=normalfont,labelsep=colon,strut=off} 
\usepackage{amsmath,amsfonts,amssymb,mathtools}
\usepackage{amsthm} 
\usepackage{bm}
\newtheorem{lemma}{Lemma} 
\newtheorem{assumption}{Assumption}
\newtheorem{theorem}{Theorem}

\usepackage{cases}
\usepackage{comment}

\title{Breaking the Convergence Barrier: Optimization via Fixed-Time Convergent Flows}
\author {
    Param Budhraja\equalcontrib,\textsuperscript{\rm 1}
    Mayank Baranwal\equalcontrib,\textsuperscript{\rm 2}
    Kunal Garg\textsuperscript{\rm 3} {\normalfont and}
    Ashish Hota\textsuperscript{\rm 1}
}
\affiliations{
    \textsuperscript{\rm 1} Indian Institute of Technology Kharagpur\\
    \textsuperscript{\rm 2} Tata Consultancy Services Research, Mumbai\\
    \textsuperscript{\rm 3} University of California, Santa Cruz\\
    budhrajap99@iitkgp.ac.in, baranwal.mayank@tcs.com, kunalgarg@ucsc.edu, ahota@ee.iitkgp.ac.in
}

\begin{document}
\maketitle
\begin{abstract}
Accelerated gradient methods are the cornerstones of large-scale, data-driven optimization problems that arise naturally in machine learning and other fields concerning data analysis. We introduce a gradient-based optimization framework for achieving acceleration, based on the recently introduced notion of fixed-time stability of dynamical systems. The method presents itself as a generalization of simple gradient-based methods suitably scaled to achieve convergence to the optimizer in a fixed-time, independent of the initialization. We achieve this by first leveraging a continuous-time framework for designing fixed-time stable dynamical systems, and later providing a consistent discretization strategy, such that the equivalent discrete-time algorithm tracks the optimizer in a practically fixed number of iterations. We also provide a theoretical analysis of the convergence behavior of the proposed gradient flows, and their robustness to additive disturbances for a range of functions obeying strong convexity, strict convexity, and possibly nonconvexity but satisfying the Polyak-Łojasiewicz inequality. We also show that the regret bound on the convergence rate is constant by virtue of the fixed-time convergence. The hyperparameters have intuitive interpretations and can be tuned to fit the requirements on the desired convergence rates. We validate the accelerated convergence properties of the proposed schemes on a range of numerical examples against the state-of-the-art optimization algorithms. Our work provides insights on developing novel optimization algorithms via discretization of continuous-time flows.
\end{abstract}


\section{Introduction and Related Work} \label{sec: Intro}

Optimization algorithms lie at the heart of modern artificial intelligence and machine learning techniques \cite{sra2012optimization}. In most applications, fast and efficient algorithms are desired for solving the optimization problem at hand. This is particularly true in machine learning applications where large data sets lead to larger problem instances and potentially larger computational time. As a result, stochastic gradient descent (SGD), its variants such as mini-batch SGD \cite{shalev2014understanding}, Adam \cite{kingma2015adam}, momentum-based, and accelerated stochastic methods have emerged as popular choices \cite{huo2018accelerated, li2020accelerated}.


In developing accelerated optimization algorithms, the discrete-time framework often proves non-intuitive and restrictive from an analytical standpoint. In contrast, continuous-time algorithms provide better intuition, and simpler and elegant proofs are often obtained by leveraging the tools of Lyapunov stability theory.
Indeed, the connection between ordinary differential equations and optimization has been recognized for several decades. For instance, \cite{brown1989} is one of the early works that leveraged the continuous-time framework to develop faster discrete-time algorithms. Similarly, the continuous-time version of gradient descent, termed {\it gradient flow} (GF) dynamics, was analyzed in \cite{su2016}. In \cite{andre2016}, the family of Bregman-Lagrangians was used to generate second-order Lagrangian flows and exponential convergence rates were established. Despite much progress, there remain two main limitations for continuous-time algorithms: (1) most of the analysis has focused on asymptotic and exponential convergence, i.e., convergence as time tends to infinity; and (2) there have been few systematic studies on developing discrete-time implementations such that the accelerated convergence properties of the continuous-time algorithm are preserved.





In this paper, we focus on continuous-time (accelerated) gradient flow dynamics with fixed-time convergence guarantees. The notion of finite-time stability (FTS), which is a precursor to the notion of fixed-time stability, was proposed in the seminal work \cite{bhat2000}. A system is said to be finite-time stable if the trajectories converge to the equilibrium in a finite amount of time, called the {\it settling time}. The settling time may depend on the initial conditions, and can potentially grow unbounded as the initial conditions go farther away from the equilibrium point. Fixed-time stability (FxTS), on the other hand, is a stronger notion, which requires the settling time to be uniformly bounded for all initial conditions, i.e., convergence within a \textit{fixed} time can be guaranteed \cite{polyakov2011nonlinear}. 

In the recent few years, continuous-time optimization methods under the notions of FTS and FxTS have gained significant interest. In \cite{cortes2006}, the author proposed a normalized version of GF and proved its finite-time stability. For convex optimization problems with equality constraints, the authors in \cite{chen2018}, designed discontinuous dynamical systems having the property of finite-time convergence. Recently in \cite{kunal2021}, FxTS gradient flows were proposed for unconstrained, constrained and min-max optimization problems. However, the aforementioned works only guarantee improved convergence in the continuous-time domain and do not provide a discrete-time implementation having accelerated convergence.

Recently, \cite{polyakov2019consistent} introduced the notion of \textit{consistent} discretization for finite, and fixed-time stable dynamical systems. In particular, they proposed an implicit discretization scheme that preserves the convergence behavior of the continuous-time system. However, these results are of little use for the optimization community, since, a) the requirement of the dynamics being homogeneous cannot be satisfied unless the equilibrium point, in this case, the optimizer, is known, and b) implicit discretization schemes are not easy to implement, thus, making it difficult to use these schemes for iterative methods. The authors in \cite{benosman2020optimizing} showed that the FTS flow, re-scaled gradient flow, and signed-gradient flow, all with a finite-time convergence, when discretized using various explicit schemes, such as Euler discretization or Runge-Kutta method, preserve the convergence behavior in the discrete-time, i.e., the minimizer could be computed within a finite number of iterations for a class of convex optimization problems. The authors evaluated their proposed methods for training neural networks and showed a significant improvement in the performance.

In this paper, a consistent discretization of FxTS-GF is proposed using the method proposed in \cite{garg2021MVIP}. We then show the robustness of FxTS-GF to vanishing disturbance, under the assumption of Polyak-Łojasiewicz (PL) inequality. Note that a function satisfying the PL-inequality can be nonconvex and that PL-inequality is a weaker assumption than strong convexity. It was shown in \cite{karimi2016linear} that PL inequality is one of the weakest assumptions under which linear convergence can be proven, which was earlier proven under the assumption of strong convexity. We then analyze the static regret of FxTS-GF and show that it is bounded by a constant. Finally, numerical experiments are conducted to compare the performance of FxTS-GF with state-of-the-art optimization algorithms. The proposed algorithm achieves lower training loss than traditional optimization methods, which also translates to better generalization on test instances.



\noindent{\bf Notation}: The set of all real numbers is denoted by $\mathbb{R}$. The set of all positive reals is denoted by $\mathbb{R}_{>0}$. The zero vector belonging to $\mathbb{R}^n$ is denoted by $\bm{0}$. For $x\in \mathbb{R}^n$, its transpose is represented by $x^\intercal$. Unless otherwise specified, $\|\cdot\|$ denotes the Euclidean norm. The set of all functions $f\!:\!U\!\to\!V$, where $U\subseteq \mathbb{R}^n$ and $V\subseteq \mathbb{R}^m$, which are $k$-times continuously differentiable is denoted by $C^k(U,V)$. The set of functions $f\!:\!U\!\rightarrow \!V$ which are continuously differentiable with locally Lipschitz continuous gradient on $U$ is denoted by $C^{loc}_{1,1}(U,V)$. For compactness, a function's argument might be omitted, whenever clear from the context. For $f\in C^1(\mathbb{R}^n, \mathbb{R})$, its gradient is denoted by $\nabla f$. A set-valued mapping $\mathcal{F}:\mathbb{R}^n \rightrightarrows \mathbb{R}^m$ maps every $x\in \mathbb{R}^n$ to a set of $\mathbb{R}^m$.


\section{Preliminaries} \label{sec: prelims}
We start by presenting the problem setting, required assumptions and the fixed-time stability property of gradient flow dynamics. Consider the unconstrained optimization problem 
\begin{equation} \label{opt_prob}
\min_{x\in \mathbb{R}^n} f(x),
\end{equation}
where $f:\mathbb{R}^n \to \mathbb{R}$. We make the following assumptions on the function $f$.

\begin{assumption} \label{basicasmp}
The function $f$ attains the minimum value $f^\star>-\infty$ at $x^\star \in \mathbb{R}^n$, i.e., $f^\star \coloneqq f(x^\star)>-\infty$.
\end{assumption}

\begin{assumption} \label{plineq}
The function $f\in C_{1,1}^{loc}(\mathbb{R}^n, \mathbb{R})$ has a unique minimizer $x=x^\star$ and satisfies Polyak-Lojasiewicz (PL) inequality, or is gradient dominated, i.e. there exists $\mu>0$, such that $\forall x\in \mathbb{R}^n$, 
\begin{equation} \label{grad_dom}
\frac{1}{2}\|\nabla f(x)\|^2 \geq \mu(f(x)-f^\star).
\end{equation}
\end{assumption}
Under the assumption of gradient dominance it was shown in (\cite{karimi2016linear}, Theorem 2) that the function $f(x)$ has a quadratic growth, i.e. 
\begin{equation} \label{quadgrowth}
f(x)-f^\star \geq \frac{\mu}{2}\|x-x^\star\|^2,
\end{equation}
for all $x\in \mathbb{R}^n$.

We now formally discuss the notion of fixed-time stability (FxTS). Consider the dynamical system 
\begin{equation} \label{dynasys}
\dot{x}=g(x), 
\end{equation}
where $x\in \mathbb{R}^n$, $g(\bm{0})=0$ and let solution to (\ref{dynasys}) exist, is unique, and continuous for any initial condition $x(0)\in \mathbb{R}^n$, for all $t\geq 0$. As introduced in \cite{polyakov2011nonlinear}, an equilibrium point of \eqref{dynasys} is called as FxTS if (i) it is Lyapunov stable, and (ii) there exists a fixed-time $T<\infty$ (also known as settling time), such that for all initial conditions $x(0)\in\mathbb R^n$, the solution of \eqref{dynasys} satisfies $x(t)\!=\!0$ for all $t\geq T$. The following lemma provides sufficient conditions for FxTS of the origin.

\begin{lemma}[\cite{polyakov2011nonlinear}] \label{fxts_lemma}
Suppose there exists a positive definite, radially unbounded function $V\in C^1(\mathcal{D},\mathbb{R})$, where $\mathcal{D}\subset \mathbb{R}^n$ is a neighbourhood of origin, such that 
\begin{equation}\label{eq:dotV cond}
\dot{V}(x)\leq -pV(x)^{\alpha} -qV(x)^{\beta}, \quad \forall x\in \mathcal{D}\setminus \{\bm{0}\}, 
\end{equation}
where $p,q >0$, $\alpha \in (0,1)$ and $\beta>1$. Then, the origin of the system (\ref{dynasys}) is fixed-time stable with a settling time 
\begin{equation}
T \leq \frac{1}{p(1-\alpha)} + \frac{1}{q(\beta -1)} . 
\end{equation}
\end{lemma}




\subsection{Fixed-Time Stable Gradient Flow (FxTS-GF)}
We now introduce the following gradient flow dynamics and establish the FxTS of the optimizer. Specifically, consider the dynamics
\begin{equation} \label{fxtsgf}
\dot{x} =\begin{cases}
-c_1\frac{\nabla f(x)}{\|\nabla f(x)\|^{\frac{p_1-2}{p_1-1}}} -c_2\frac{\nabla f(x)}{\|\nabla f(x)\|^{\frac{p_2-2}{p_2-1}}}, & \nabla f(x) \neq \bm{0},\\
0, & \nabla f(x) = \bm{0},
\end{cases} 
\end{equation}
where $c_1,c_2 >0$, $p_1>2$ and $p_2 \in (1,2)$. The flow \eqref{fxtsgf}, first introduced in \cite{kunal2021}, is henceforth called as FxTS-GF. In \cite{kunal2021}, it was shown that the flow \eqref{fxtsgf} converges to the optimizer $x^\star$ of \eqref{opt_prob} within a fixed time, irrespective of the initial condition, under the Assumptions \ref{basicasmp} and \ref{plineq}. We now provide a self-contained proof of FxTS of \eqref{fxtsgf} with a different candidate Lyapunov function.


\begin{theorem}[\textbf{FxTS-GF}]\label{lemma:fxts_V}
Suppose the function $f$ satisfies Assumptions~\ref{basicasmp} and \ref{plineq}. Then, the flow given by \eqref{fxtsgf} converges to the optimizer $x^\star$ in a fixed time for all $x(0)\in \mathbb R^n$. 
\end{theorem}
\begin{proof}
The existence and uniqueness of a continuous solution of \eqref{fxtsgf} for initial conditions at all times was proved in \cite[Proposition 1]{kunal2021}. Thus, we proceed with the proof of FxTS of the optimal point $x^\star$. Consider Lyapunov candidate $V(x)=f(x)-f^\star$. As $f^\star$ is the minimum value of $f$ and $x^\star$ is the unique minimizer of $f$, it holds that $V(x)> 0$ for all $x\neq x^\star$. The time derivative of $V$ is given by 
\begin{align*}
\dot{V}(x)&=\nabla f(x)^\intercal  \dot{x} \\
&= \nabla f(x)^\intercal \left(\!-c_1\frac{\nabla f(x)}{\|\nabla f(x)\|^{\frac{p_1-2}{p_1-1}}} -c_2\frac{\nabla f(x)}{\|\nabla f(x)\|^{\frac{p_2-2}{p_2-1}}}\!\right) \\
&= -c_1 \|\nabla f(x)\|^{\frac{p_1}{p_1 -1}} -c_2 \|\nabla f(x)\|^{\frac{p_2}{p_2 -1}}.
\end{align*}
Using inequality (\ref{grad_dom}), we get 
\begin{align*}
\dot{V} &\leq \!-c_1(2\mu(f\!-\!f^\star))^{\frac{p_1}{2(p_1\!-\!1)}}\!-\!c_2(2\mu(f\!-\!f^\star))^{\frac{p_2}{2(p_2\!-\!1)}} \\
&= -c_1\!(2\mu)^{\frac{p_1}{2(p_1-1)}}V^{\frac{p_1}{2(p_1-1)}}\!-\!c_2(2\mu)^{\frac{p_2}{2(p_2-1)}} V^{\frac{p_2}{2(p_2-1)}}.
\end{align*}
Define $p\coloneqq c_1 (2\mu)^{\frac{p_1}{2(p_1-1)}}$, $q\coloneqq c_2 (2\mu)^{\frac{p_2}{2(p_2-1)}}$. Since $p_1>2$ and $p_2\in (1,2)$, we have $\alpha\coloneqq\frac{p_1}{2(p_1-1)} \in (0,1)$ and $\beta\coloneqq\frac{p_2}{2(p_2-1)} > 1$. Thus, the conditions for Lemma \ref{fxts_lemma} are satisfied, and it follows that the equilibrium point $x^\star$ of \eqref{fxtsgf} is FxTS with settling time $T\leq \frac{1}{p(1-\alpha)} + \frac{1}{q(\beta -1)}$.
\end{proof}

\section{Robustness Analysis} 
The above result establishes the FxTS of the optimizer under flow dynamics \eqref{fxtsgf}. We now prove a stronger result that the FxTS property is preserved when the dynamics \eqref{fxtsgf} is subjected to additive noises or disturbances, under mild assumptions on the noise or the disturbances. This is particularly important in data-driven learning where only a noisy estimate of the gradient is available. Specifically, we consider the following dynamical system:
\begin{equation} \label{noisyfxts}
\dot{x}= -c_1\frac{\nabla f(x)}{\|\nabla f(x)\|^{\frac{p_1-2}{p_1-1}}} -c_2\frac{\nabla f(x)}{\|\nabla f(x)\|^{\frac{p_2-2}{p_2-1}}} + \varepsilon (x),
\end{equation}
where $\varepsilon:\mathbb
R^n\rightarrow\mathbb R^n$ is the additive noise term. We assume that the noise $\varepsilon(\cdot)$ is a vanishing disturbance and it satisfies the following assumption.
\begin{assumption} \label{as:noiseterm}
    There exists $l>0$ such that the noise term satisfies $\|\varepsilon (x)\|\leq l\|x-x^\star\|^2$, for all $x\in\mathbb{R}^n$. 
\end{assumption}

\noindent Observe that Assumption~\ref{plineq} along with \eqref{quadgrowth} yield 
\begin{equation*}
    \|x-x^\star\|^2 \leq \frac{1}{\mu^2}\|\nabla f(x)\|^2.
\end{equation*}
Together with Assumption \ref{as:noiseterm}, we obtain the following bound on the noise $\varepsilon$:
\begin{equation} \label{eq: noisebound2}
\|\varepsilon(x)\| \leq  \bar{l}\|\nabla f(x)\|^2 ,
\end{equation}
where $\bar{l}=\frac{l}{\mu^2}$. Note that Assumption \ref{as:noiseterm} also implies that the point $x^\star$ is an equilibrium point of the perturbed flow in \eqref{noisyfxts}. We now show that the FxTS property of $x^\star$ is robust to additive disturbance satisfying Assumption \ref{as:noiseterm}. 




\begin{theorem}[\textbf{Robustness of FxTS-GF}]\label{thm:Robustness}
Under Assumptions \ref{basicasmp}, \ref{plineq}, \ref{as:noiseterm}, if $c_1, c_2, p_2$ are chosen so that 
$4\mu^2 \min \left\{c_1, c_2 \right\}>l$ and $1<p_2\leq \frac{3}{2}$,
then the equilibrium point $x^\star$ is FxTS for the flow given by \eqref{noisyfxts}.
\end{theorem}

\begin{proof}
Consider the Lyapunov candidate $V(x)=f(x)-f^\star$ as before. The time derivative of $V$ is given by 
\begin{align*}
\dot{V}(x) = -c_1\|\nabla f(x)\|^{\frac{p_1}{p_1-1}}-c_2\|\nabla f(x)\|^{\frac{p_2}{p_2-1}}+\nabla f^\intercal \varepsilon(x). 
\end{align*} 
We define the following constants for the ease notation: $\alpha_1= \frac{p_1}{p_1 -1}$, $\beta_1= \frac{p_2}{p_2 -1}$. Note that $\alpha_1 \in (0,2)$ and $\beta_1 >2$. 
Now, we prove that 
\begin{equation} \label{eq:robust_bound}
\dot{V}(x)\leq -(c_1 -\bar{l})\|\nabla f(x)\|^{\alpha_1} -(c_2 -\bar{l})\|\nabla f(x)\|^{\beta_1},
\end{equation}
for all $x\in \mathbb{R}^n$. 
Using Assumption \ref{as:noiseterm} and triangle inequality, we obtain that $\nabla f^\intercal \varepsilon(x)\leq \|\nabla f(x)\|\|\varepsilon(x)\|\leq \bar{l}\|\nabla f(x)\|^3$. Thus, it follows that
\begin{equation}\label{eq: dot V robust proof}
\dot{V}(x) \leq -c_1 \|\nabla f(x)\|^{\alpha_1} -c_2\|\nabla f(x)\|^{\beta_1} +\bar{l}\|\nabla f(x)\|^3.
\end{equation}
Define $S = \{x\; |\; \|\nabla f(x)\|\leq 1\}$ and consider the two cases, namely, when $x\in S$ and $x\notin S$. 

First, consider the case when $x\notin S$. Re-arranging the right-hand side of \eqref{eq: dot V robust proof}, we obtain: 
\begin{align*}
\dot{V}(x) &\leq -c_1\|\nabla f(x)\|^{\alpha_1} -(c_2-\bar{l})\|\nabla f(x)\|^{\beta_1} \\ &\quad +\bar{l}\left(\|\nabla f(x)\|^3 -\|\nabla f(x)\|^{\beta_1}\right).
\end{align*}
Since $p_2\in \left(1,\frac{3}{2}\right]$, it follows that $\beta_1>3$. Furthermore, for all $x\notin S$, it holds that $\|\nabla f(x)\|>1$ and thus, it follows that $\|\nabla f(x)\|^3 -\|\nabla f(x)\|^{\beta_1}\leq0$. Hence, we obtain that
\begin{align*}
\dot{V}(x) &\leq -c_1\|\nabla f(x)\|^{\alpha_1} -(c_2-\bar{l})\|\nabla f(x)\|^{\beta_1} \\
&\leq -(c_1 -\bar{l})\|\nabla f(x)\|^{\alpha_1} -(c_2-\bar{l})\|\nabla f(x)\|^{\beta_1}.
\end{align*}

Next, consider the case when $x\in S$. Re-arranging the right-hand side of \eqref{eq: dot V robust proof}, we obtain
\begin{align*}
\dot{V}(x) &\leq -(c_1 -\bar{l})\|\nabla f(x)\|^{\alpha_1} -c_2\|\nabla f(x)\|^{\beta_1} \\ &\quad
+\bar{l}\left(\|\nabla f(x)\|^3 - \|\nabla f(x)\|^{\alpha_1} \right).
\end{align*}
For all $x\in S$, we have $\|\nabla f(x)\|^3 - \|\nabla f(x)\|^{\alpha_1} \leq 0$ since $\alpha_1 \in (0,2)$ and $\|\nabla f(x)\|\leq 1$. Thus, it holds that
\begin{align*}
\dot{V}(x) &\leq -(c_1 -\bar{l})\|\nabla f(x)\|^{\alpha_1} -c_2\|\nabla f(x)\|^{\beta_1} \\ 
&\leq -(c_1 -\bar{l})\|\nabla f(x)\|^{\alpha_1} -(c_2-\bar{l})\|\nabla f(x)\|^{\beta_1}.
\end{align*}
Thus, it follows from Assumption~\ref{grad_dom} that for all $x\in \mathbb R^n$
\begin{align*}
\dot{V}(x) &\leq -(c_1-\bar{l})(2\mu(f(x)-f^\star))^{\frac{p_1}{2(p_1-1)}} \\
&\qquad -(c_2-\bar{l})(2\mu(f(x)-f^\star))^{\frac{p_2}{2(p_2-1)}}.
\end{align*}
Define $p := (c_1-\bar{l}) (2\mu)^{\frac{p_1}{2(p_1-1)}}, q := (c_2-\bar{l}) (2\mu)^{\frac{p_2}{2(p_2-1)}},$ and $\alpha:= \frac{p_1}{2(p_1-1)}, \beta:=\frac{p_2}{2(p_2-1)}$ so that we have 
\[\dot{V}(x) \leq -p V(x)^{\alpha}-q V(x)^{\beta},\]
where $p,q>0$, $\alpha \in (0,1)$ and $\beta>1$. Using Lemma \ref{fxts_lemma}, we obtain that the equilibrium point $x^\star$ is fixed-time stable for the perturbed FxTS-GF flow given by \eqref{noisyfxts}.
\end{proof}

Thus, FxTS-GF flow in \eqref{fxtsgf} is robust against a class of vanishing additive disturbances.

\section{Regret Analysis} \label{sec: regret analysis}
The regret analysis is used to evaluate the effectiveness of an algorithm~\cite{sun2020continuous}. In this section, we analyze the regret of FxTS-GF \eqref{fxtsgf} in the offline setting. The static regret is defined as the accumulated difference between the objective function computed according to the state of the algorithm and the objective function computed according to the best fixed-point that minimizes the accumulated objective function. The regret at any time $T\!>\!0$ is given by
\begin{equation*}
    \mathcal{R}_S (T,x_0)= \int_0^T \left( f(x(t)) -f(x^\star)\right)dt,
\end{equation*}
where $x(0)=x_0$ is the initial condition. Note that in the offline setting the dynamic and static regret are equivalent. Also, observe that static regret is the time integral of the Lyapunov candidate used in the Theorem~\ref{lemma:fxts_V}.
\begin{theorem}[Regret Bound]\label{thm:Regret}
    Under the Assumptions \ref{basicasmp} and \ref{plineq}, the static regret of the flow FxTS-GF is bounded by a constant $l_1 +l_2$, with 
    \begin{equation*}
    l_1 = \frac{1}{p(2-\alpha)}, \quad l_2= \frac{(1+V(0)^{\beta-1})^{\frac{\beta-2}{\beta-1}}-1}{qV(0)^{\beta-2}(\beta-2)},
    \end{equation*}
    where $V(0)=f(x_0)-f^\star$.
\end{theorem}
\begin{proof}
Considering the Lyapunov candidate $V(x)=f(x)-f^\star$, we have the following bound on its time derivative $\dot{V}(x)$, which was proved in Theorem~\ref{lemma:fxts_V}, \begin{equation}
\dot{V}(x) \leq -pV(x)^{\alpha} -qV(x)^{\beta},
\end{equation}
where $p=c_1(2\mu)^{\frac{p_1}{2(p_1 -1)}},\: q=c_1(2\mu)^{\frac{p_2}{2(p_2 -1)}},\: \alpha= \frac{p_1}{2(p_1 -1)}$, and $\beta= \frac{p_2}{2(p_2 -1)}$.
When $V(t)> 1$ we use $\dot{V}(t)\leq -qV(t)^{\beta}$ and when $V(t)\leq 1$ we use $\dot{V}(t)\leq -pV(t)^{\alpha}$. The main idea is to apply these two approximations to obtain upper bound on $V(t)$, which in turn bounds the regret. We define the following constants:
\begin{align*}
T_1 &:= \frac{1}{p(1-\alpha)}, \quad T_2:= \frac{1}{q(\beta-1)}.
\end{align*}
We divide the proof into two parts. First we analyze the case $V(x(0))>1$ and then the case $V(x(0))\leq 1$. \\
\textbf{Case 1:} Consider $V(x(0))>1$, i.e. the initial conditions $x_0$ is such that $f(x_0)-f^\star>1$. For $t\leq T_2$ we have 
\begin{equation} \label{vbound1}
V(x(t))\leq \frac{V(x(0))}{\left( 1+qV(x(0))^{\beta-1}(\beta-1)t\right)^{\frac{1}{\beta-1}}}\cdot
\end{equation}
For $T_2 \leq t\leq T_1+T_2$, we get
\begin{equation} \label{vbound2}
V(x(t))\leq \left(1-p(1-\alpha)(t-T_2)\right)^{\frac{1}{1-\alpha}}.
\end{equation}
For $t\geq T_1+T_2$, $V(x(t))=0$, i,e, the solution converges to optimizer. We integrate both sides of the inequality (\ref{vbound1}) to get the following for all $T\in[0,T_1]$:
\begin{equation*}
\mathcal{R}_S(T,x_0)\leq \frac{(1+q(\beta-1)V(x(0))T)^{\frac{\beta-2}{\beta-1}}-1}{q(\beta-2)V(x(0))^{\beta-2}}.
\end{equation*}
Similarly for $T_2\leq T\leq T_1+T_2$, using both the inequalities (\ref{vbound1}) and (\ref{vbound2}), we get
\begin{equation*}
\mathcal{R}_S(T,x_0)\leq l_2 + \frac{1-\left(1-p(1-\alpha)(T-T_2) \right)^{\frac{2-\alpha}{1-\alpha}}}{p(2-\alpha)} .
\end{equation*}
As $V(t)=0$ for $t\geq T_1+T_2$, we get $\mathcal{R}_S(T,x_0)\leq l_1+l_2$ for $T\geq T_1+T_2$. \\
\textbf{Case 2:} Consider $V(x(0))\leq1$, i.e. the initial conditions $x_0$ is such that $f(x_0)-f^\star\leq1$. In this scenario also we use the same procedure. For $0\leq T\leq T_1$ we get
\begin{equation*}
\mathcal{R}_S(T,x_0)\leq  \frac{1-\left(1-p(1-\alpha)T \right)^{\frac{2-\alpha}{1-\alpha}}}{p(2-\alpha)} .
\end{equation*}
For $T\geq T_1$ we get $\mathcal{R}_S (T,x_0) \leq l_1$. Note that we can also say that $\mathcal{R}_S(T,x_0)\leq l_1+l_2$ for all $T\geq0$ and for all $V(0)$.
\end{proof}
Observe that instead of bounding by a constant we can also bound the regret $\mathcal{R}_S(T,x_0)$ by a step function as: 
\begin{align*}
\mathcal{R}_S(T,x_0)= \begin{cases} 
l_1 \quad &\text{if, } f(x_0)-f^\star \leq1, \: T\geq 0,\\ 
l_1 \quad &\text{if, } f(x_0)-f^\star >1, \: T\leq T_1,\\
l_1+l_2 \quad &\text{if, } f(x_0)-f^\star >1, \: T> T_1.
\end{cases}
\end{align*}


\section{Discretization of FxTS-GF} \label{sec: discretization} 
In practice, continuous-time dynamical systems can be implemented using iterative discrete-time approximations.
Following the work in \cite{garg2021MVIP,benosman2020optimizing}, 
in this section, we provide the analysis of the Euler-discretization of \eqref{fxtsgf}, and show that when the FxTS-GF~\eqref{fxtsgf} is discretized using Euler discretization, it leads to a consistent discretization. We use the following result from \cite{garg2021MVIP}.

\begin{lemma}[\textbf{Consistent Discretization}]\label{lemma: weak conv disc FxTS}
    Consider the following differential inclusion:
    \begin{equation}\label{eq:discrete cont dyn}
        \dot x \in \mathcal F(x),
    \end{equation}
    where $\mathcal{F}: \mathbb{R}^n\rightrightarrows\mathbb{R}^n$ is an upper semi-continuous set-valued map, taking non-empty, convex and compact values, with $0\in\mathcal{F}(\bar{x})$ for some $\bar x\in \mathbb R^n$. Assume that there exists a positive definite, radially unbounded $V: \mathbb R^n\rightarrow\mathbb R$ such that $V(\bar{x}) = 0$ satisfying \eqref{eq:dotV cond} with $\alpha = 1-\frac{1}{\xi}, \beta = 1+\frac{1}{\xi}$ for some $\xi>1$.
    If the function $V$ satisfies $V(x)\geq m\|x-\bar{x}\|^2$ for all $x\in\mathbb{R}^n$, where $m>0$ and $\bar x$ is the equilibrium point of \eqref{eq:discrete cont dyn}, then, for all $x_0\in\mathbb{R}^n$ and $\epsilon>0$, there exists $\eta^*>0$ such that for any $\eta\in (0,\eta^*]$, the following holds:{\small
    \begin{align}\label{eq:x disc bound}
        \hspace{-5pt}\|x_k-\bar{x}\|<
        \begin{cases}
        \frac{1}{\sqrt{m}}\left(\sqrt{\frac{p}{q}}\tan\left(\frac{\pi}{2}-\frac{\sqrt{pq}}{\xi}\eta k\right)\right)^\frac{\xi}{2}+\epsilon, & k\leq k^*;\\
        \epsilon, & k>k^*,
        \end{cases}
    \end{align}}\normalsize
    where $k^* = \Big\lceil\frac{\xi\pi}{2\eta\sqrt{pq}}\Big\rceil$ and $x_k$ is a solution of the forward-Euler discretization of \eqref{eq:discrete cont dyn}:
\begin{equation}\label{eq:discrete dyn FxTS}
    x_{k+1} \in x_k + \mathcal \eta \mathcal F(x_k),
\end{equation}
where $\eta>0$ is the time-step, starting from the point $x_0$.
\end{lemma}

Thus, in order to prove that an Euler discretization scheme of \eqref{fxtsgf} leads to a consistent discretization, it is sufficient to show that \eqref{fxtsgf} satisfies the conditions of Lemma~\ref{lemma: weak conv disc FxTS}.

\begin{lemma}\label{lemma:V p1 p2}
If $p_1, p_2$ satisfy
\begin{align}\label{eq: p1 p2 cond}
 2 + \frac{1}{p_1-2}  = \frac{1}{2-p_2},
\end{align}
with $\frac{3}{2}<p_2<2$, then, the function $V(x) = (f(x)-f^\star)$ satisfies conditions of Lemma \ref{lemma: weak conv disc FxTS}.
\end{lemma}
\begin{proof}
Consider the Lyapunov candidate $V(x) = (f(x)-f^\star)$. Its time derivative along the trajectories of \eqref{fxtsgf} reads:
\begin{align*}
    \dot V(x) =& -c_1\nabla f(x)^\intercal \frac{\nabla f(x)}{\|\nabla f\|^{\frac{p_2-2}{p_2-1}}}-c_2\nabla f(x)^\intercal \frac{\nabla f(x)}{\|\nabla f\|^{\frac{p_1-2}{p_1-1}}}.
\end{align*}
Following the analysis in Theorem~\ref{lemma:fxts_V}, it follows that
\begin{align*}
\dot{V} &\leq -pV^{\frac{p_1}{2(p_1-1)}} -q V^{\frac{p_2}{2(p_2-1)}},
\end{align*}
where $p= c_1 (2\mu)^{\frac{p_1}{2(p_1-1)}}$ and $q=c_2 (2\mu)^{\frac{p_2}{2(p_2-1)}}$. Note that under the condition \eqref{eq: p1 p2 cond}, it holds that there exists $\xi = -\frac{2p_2-2}{p_2-2} = \frac{2p_1-2}{p_1-2}>2$, so that the above equation reads
\begin{align*}
    \dot V(x) = -pV(x)^{1+\frac{1}{\xi}}
    -qV(x)^{1-\frac{1}{\xi}}.
\end{align*}
Thus, the candidate function $V$ satisfies the conditions of Lemma \ref{fxts_lemma} with $\alpha = 1-\frac{1}{\xi}$ and $\beta = 1+\frac{1}{\xi}$. Finally, note that under Assumption \ref{plineq}, it holds that $\frac{1}{2}\|\nabla f(x)\|^2\geq \mu (f(x)-f^\star)\geq 2\mu^2\|x-x^\star\|^2$, i.e., the function $V$ has quadratic growth, and thus, the function $V$ is radially unbounded, satisfying all the conditions of Lemma \ref{lemma: weak conv disc FxTS}. 
\end{proof}

\begin{theorem}\label{thm: discretized}
Assume that the functions $f$ satisfy Assumptions~\ref{basicasmp}-\ref{plineq}. Consider the discrete-time system 
\begin{align}\label{eq: dot x disc}
    x_{k+1} = x_k -\eta c_1\frac{\nabla f(x_k)}{\|\nabla f(x_k)\|^{\frac{p_1-2}{p_1-1}}} -\eta c_2\frac{\nabla f(x_k)}{\|\nabla f(x_k)\|^{\frac{p_2-2}{p_2-1}}}, 
\end{align}
obtained from discretizing the dynamics in \eqref{fxtsgf} using Euler's method with time step $\eta>0$, where $p_1, p_2$ satisfy \eqref{eq: p1 p2 cond}. Then, for all $\epsilon>0$, there exists $\eta^*>0$ such that for all $\eta\in (0, \eta^*]$, the trajectories of \eqref{eq: dot x disc} satisfy {\small
\begin{align}\label{eq: disc bound}
    \|x_k-x^\star\|\leq \begin{cases}
   \frac{1}{\sqrt{2}\mu}\left(\!\sqrt{\frac{p}{q}}\tan\left(\frac{\pi}{2}\!-\!\frac{\eta k\sqrt{pq}}{2\mu}\!\right)\right)^{\mu}\!+\!\epsilon\; ; & k\leq k^*,\\
    \epsilon \; ; & k > k^*,
    \end{cases}
\end{align}}\normalsize
where $k^* = \lceil \frac{\mu\pi}{\sqrt{pq}\eta} \rceil$, and $a, b, c_1, \mu>0$.
\end{theorem}
\begin{proof}
The proof is based on Lemma \ref{lemma: weak conv disc FxTS}. First, note that per Lemma \ref{lemma:V p1 p2}, there exists a function $V$, namely $V(x) = (f(x)-f^\star)$, that satisfies the conditions of Lemma \ref{lemma: weak conv disc FxTS} with $\xi = \frac{2p_1-2}{p_1-2}$ and $\beta = 2\mu^2$. Next, note that the right-hand side of \eqref{fxtsgf} is single-valued and continuous, and thus, $\mathcal F(x) = -c_1\frac{\nabla f(x)}{\|\nabla f(x)\|^{\frac{p_1-2}{p_1-1}}} -c_2\frac{\nabla f(x)}{\|\nabla f(x)\|^{\frac{p_2-2}{p_2-1}}}$ satisfies the assumptions of Lemma \ref{lemma: weak conv disc FxTS}  with $\bar x = x^\star$. Thus, it holds that all the conditions of Lemma \ref{lemma: weak conv disc FxTS} are satisfied and the proof is complete. 
\end{proof}


\section{FxTS Gradient Flow with Momentum} \label{sec:FxTS-M-GF}
Training of neural networks entails computing gradients on mini-batches. These gradients are not exact and serve as only \emph{noisy} estimates of the true gradient of the loss function, leading to optimization algorithms not descending in optimal directions. This can be partially alleviated using the momentum method~\cite{polyak1964some}, which employs exponentially weighted averages to provide a better estimate of the true gradient. Gradient descent with momentum is defined by:
\begin{align}\label{eq:SGD-momentum}
    v_{t} &= \beta v_{t-1} + (1-\beta)\nabla f(x_{t-1})\nonumber\\
    x_{t} &= x_{t-1} - \alpha v_{t},
\end{align}
where $\alpha>0$ and $\beta\in[0,1)$. Here $x_t$ represents the $t^{\text{th}}$-iterate of the state $x$. In the limit that $\alpha$ is sufficiently small, the derivative $\dot{x}(t)$ can be approximated as $\dot{x}(t)\approx (x_t-x_{t-1})/\alpha$. Using this analogy, the continuous-time variant of the momentum method can be expressed as:
\begin{align}\label{eq:SGD-momentum-cont}
    \dot{x}(t) &= - v(t) \nonumber\\
    \dot{v}(t) &= \lambda (\nabla f(x(t))-v(t)),
\end{align}
where $\lambda = (1-\beta)/\alpha$. We now propose a suitable modification of \eqref{eq:SGD-momentum-cont}, such that the resulting dynamics converges to equilibrium $(x^\star,\bm{0})$ in a fixed-time. We refer to this modified fixed-time stable dynamical system as \emph{FxTS(M)-GF} which is an acronym for fixed-time stable gradient flow with momentum. The continuous-time dynamics for the FxTS(M)-GF for $(x,v)\neq(x^\star,\bm{0})$ is described as:
\begin{align}\label{eq:FxTS-M-GF}
    \dot{x} &= - v\cdot h(x,v)  \nonumber\\
    \dot{v} &= \lambda (\nabla f(x)-v)\cdot g_{p,q}(\|\nabla\!f(x)\!-\!v\|),
\end{align}
where $g_{p,q}:\mathbb R_{>0}\!\to\!\mathbb R_{>0}$ is defined as $g_{p,q}(s)\coloneqq 1\big/{(s)^{\frac{p-2}{p-1}}}+1\big/{(s)^{\frac{q-2}{q-1}}}$ with $p>2$, $q\in(1,2)$. The function $h(x,v)$ is defined as:
\begin{equation*}
    h(x,v) = \left\{\begin{array}{cl}
         \!\!\!g_{p,q}(\|\nabla\!f(x)\|), & \text{if $\|\nabla\!f(x)\|>\|\nabla\!f(x)\!-\!v\|$} \\
         \!\!\!1, & \text{else}
    \end{array}\right.\!.
\end{equation*}
It is possible to choose the exponents $p, q$ such that $g_{p,q}$ is monotonically decreasing on $\mathbb R_{>0}$. For instance, choosing $p=2.1$ and $q=1.98$ results in a monotonically decreasing $g_{p,q}(\cdot)$. We now state our main result.

\begin{assumption}\label{ass:LipGradient}
    The function $f$ is $\mu$-strongly convex and has an $L$-Lipschitz continuous gradient, i.e., $\mu\leq\|\nabla^2 f(x)\|\leq L$ for all $x\in\mathbb{R}^n$.
\end{assumption}

\begin{theorem}\label{thm:FxTS-M-GF}
    Under Assumptions~\ref{basicasmp} and \ref{ass:LipGradient}, if exponents $p$ and $q$ are chosen such that $g_{p,q}(\cdot)$ is monotonically decreasing on $\mathbb R_{>0}$ and $\lambda>L$, then the equilibrium point $(x^\star,\bm{0})$ is fixed-time stable for the flow described in \eqref{eq:FxTS-M-GF}.
\end{theorem}
\begin{proof}
    We consider the candidate Lyapunov function,
    \begin{equation}\label{eq:FxTS-M-GF-V}
        V(x,v) = \underbrace{\frac{1}{2}\|\nabla f(x)\|^2}_{V_1(x)} + \underbrace{\frac{1}{2}\|\nabla f(x)-v\|^2}_{V_2(x,v)}.
    \end{equation}
    Clearly, $V(x,v)>0$ for all $(x,v)\neq(x^\star,\bm{0})$. Additionally, define a set $\mathcal{S} = \{(x,v)\; |\; \|\nabla f(x)\|\leq \|\nabla f(x)-v\|\}$, and observe that
    \begin{align}\label{eq:FxTS-M-GF-V1V2}
        x\in \mathcal S \quad \implies V &\leq 2V_2\nonumber\\
        x\notin \mathcal S \quad \implies V &\leq 2V_1.
    \end{align}
Taking time derivative of $V$ along trajectories of \eqref{eq:FxTS-M-GF} yields
    \begin{align*}
        \dot{V} &= (\nabla\!f)^\intercal(\nabla^2\!f)\dot{x} + (\nabla\!f-v)^\intercal((\nabla^2\!f)\dot{x}-\dot{v})\\
        &= -h\!\cdot\!(\nabla\!f)^\intercal(\nabla^2\!f)v-h\!\cdot\!(\nabla\!f-v)^\intercal(\nabla^2\!f)v\\
        &\quad -\lambda(\nabla\!f-v)^\intercal(\nabla\!f-v)\!\cdot\!g_{p,q}(\|\nabla\!f-v\|)\\
        &= -h\!\cdot\!(\nabla\!f)^\intercal(\nabla^2\!f)(\nabla\!f) + h\!\cdot\!(\nabla\!f-v)^\intercal(\nabla^2\!f)(\nabla\!f-v)\\
        &\quad -\lambda(\nabla\!f-v)^\intercal(\nabla\!f-v)\!\cdot\!g_{p,q}(\|\nabla\!f-v\|)\\
        &\leq -h\!\cdot\!\mu\|\nabla f\|^2 +\|\nabla f\!-\!v\|^2\left(h\!\cdot\!L\!-\!\lambda\!\cdot\!g_{p,q}(\|\nabla\!f\!-\!v\|)\right),
    \end{align*}
    where the last inequality follows from $\mu$-strong convexity and $L$-Lipschitz gradient conditions.\\
    {\bf Case 1: $x\notin \mathcal S$}: In this case, $h$ is given by $g_{p,q}(\|\nabla\!f\|)$. Moreover, $\|\nabla\!f\!-\!v\|\leq \|\nabla f\|$, which leads to $g_{p,q}(\|\nabla\!f\|)\leq g_{p,q}(\|\nabla\!f-v\|)$ due to the monotonicity of $g_{p,q}$. Using this and the fact that $\lambda>L$, $\dot{V}(x)$ can be upper-bounded as:
    \begin{align}\label{eq:V1dot}
        \dot{V} &\leq \!-g_{p,q}(\|\!\nabla\!f\!\|)\mu\|\nabla\!f\|^2\!-\! g_{p,q}(\|\!\nabla\!f\!-\!v\|)\!(\lambda\!-\!L)\!\|\nabla\!f\!-\!v\|^2\nonumber\\
        &\leq -\mu\left(2{V_1}\right)^{\frac{p}{2(p-1)}}-\mu\left(2{V_1}\right)^{\frac{q}{2(q-1)}}, \nonumber\\
        &\leq -\mu V^{\frac{p}{2(p-1)}}-\mu V^{\frac{q}{2(q-1)}}.
    \end{align}
    {\bf Case 2: $x\in \mathcal S$}: In this case, $h=1$ and thus, we have
    \begin{align}\label{eq:V2dot}
        \dot{V} &\leq -\mu\|\nabla\!f\|^2 + \left(L-\lambda\!\cdot\!g_{p,q}(\|\nabla\!f\!-\!v\|)\right)\|\nabla\!f\!-\!v\|^2\nonumber\\
        &\leq L\left(2{V_2}\right) - \lambda\left(2{V_2}\right)^{\frac{p}{2(p-1)}}- \lambda\left(2{V_2}\right)^{\frac{q}{2(q-1)}}\nonumber\\
        &\leq 2LV - \lambda V^{\frac{p}{2(p-1)}}- \lambda V^{\frac{q}{2(q-1)}}.
    \end{align}
    Since $p>2$ and $q\in(1,2)$, using the similar arguments as in the proof of Theorem~\ref{thm:Robustness}, we obtain that there exists $\gamma>0$ such that for all $x\in \mathcal S$
    \begin{equation}\label{eq:V2dot2}
        \dot{V}(x) \leq -\gamma V(x)^{\frac{p}{2(p-1)}}-\gamma V(x)^{\frac{q}{2(q-1)}}.
    \end{equation}
    Thus, from \eqref{eq:V1dot} and \eqref{eq:V2dot2}, it follows that
    \begin{equation}\label{eq:finalVdot}
        \dot{V}(x) \leq -\min\{\gamma,\mu\}\left(V(x)^{\frac{p}{2(p-1)}}+V(x)^{\frac{q}{2(q-1)}}\right),
    \end{equation}
    for all $x\in \mathbb R^n$
    i.e., the FxTS(M)-GF is fixed-time convergent gradient flow following Lemma \ref{lemma:fxts_V}.
\end{proof}\vspace{-.5em}
The proposed FxTS(M)-GF inherits some of the desirable properties of the aforementioned FxTS-GF, such as robustness and constant regret, however, a detailed analysis of the FxTS(M)-GF is left for future work.

\begin{figure}
     \centering
     \begin{subfigure}
         \centering
         \includegraphics[width=0.47\linewidth]{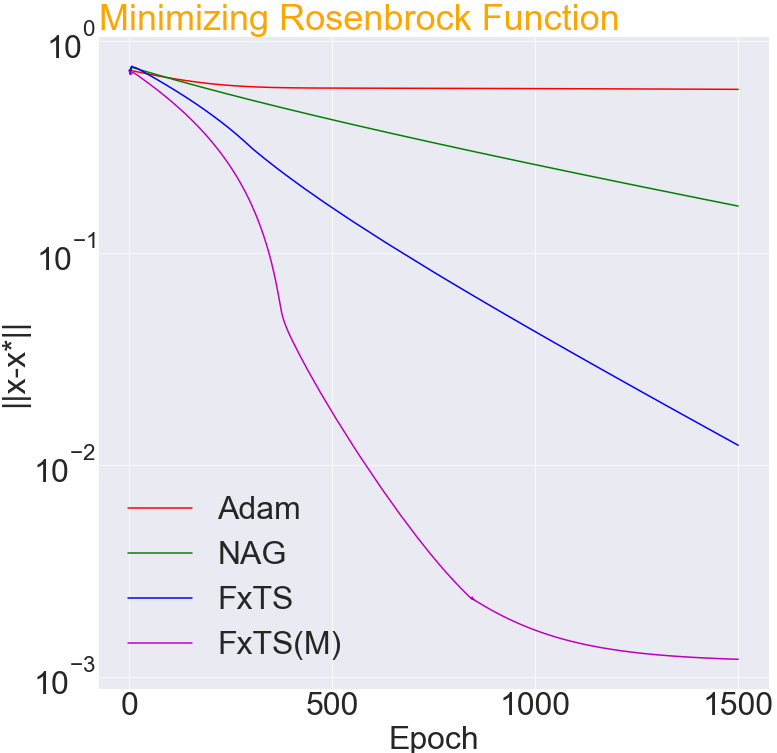}
         \label{fig:sub-first}
     \end{subfigure}
     \hfill
     \begin{subfigure}
         \centering
         \includegraphics[width=0.47\linewidth]{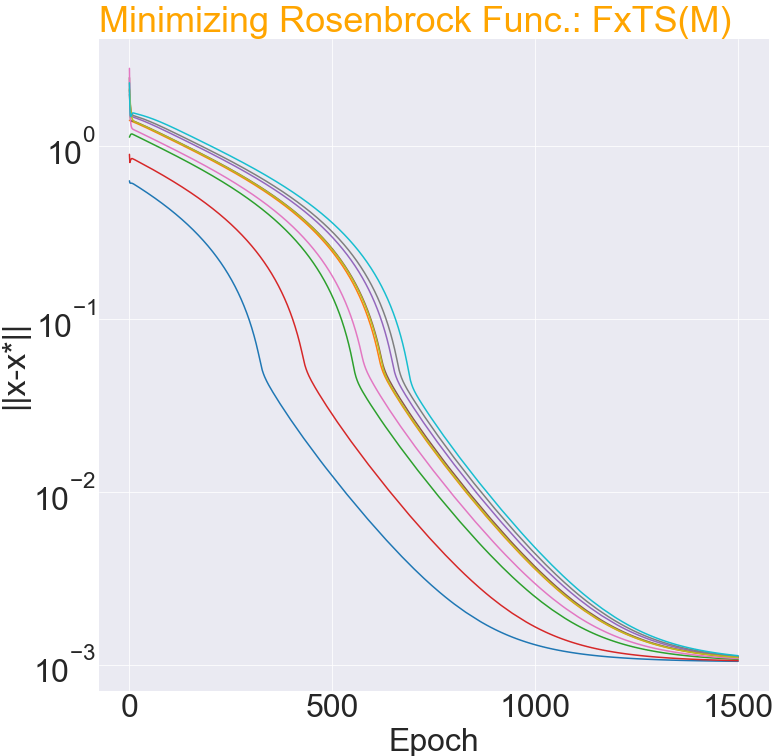}
         \label{fig:sub-second}
     \end{subfigure}
        \caption{Minimization of Rosenbrock function. (a) Comparison of various optimization algorithms for the initial condition $(0.3,0.8)$. (b) Performance of the FxTS(M)-GF algorithm at varying initial conditions.}
        \label{fig:Rosenbrock}
        \vspace{-1em}
\end{figure}

\begin{figure*}[!ht]
	\begin{center}
		\begin{tabular}{ccc}
			\includegraphics[width=0.5\columnwidth]{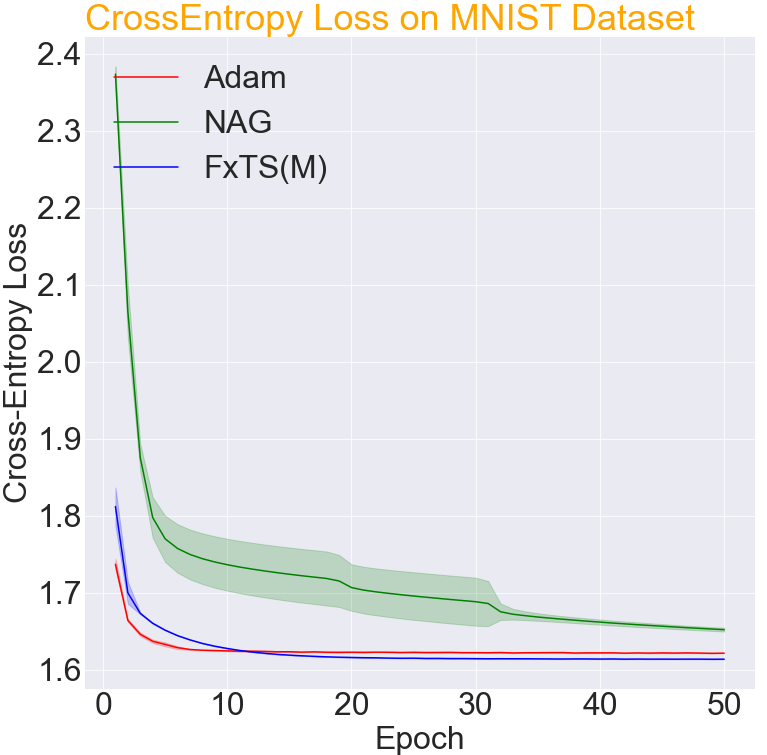}&
			\includegraphics[width=0.5\columnwidth]{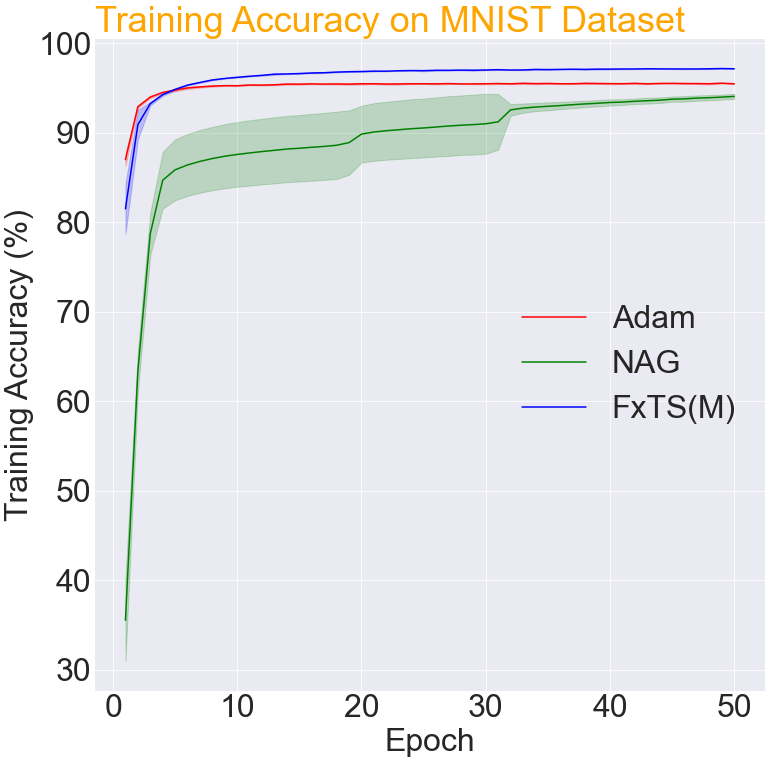}&
			\includegraphics[width=0.5\columnwidth]{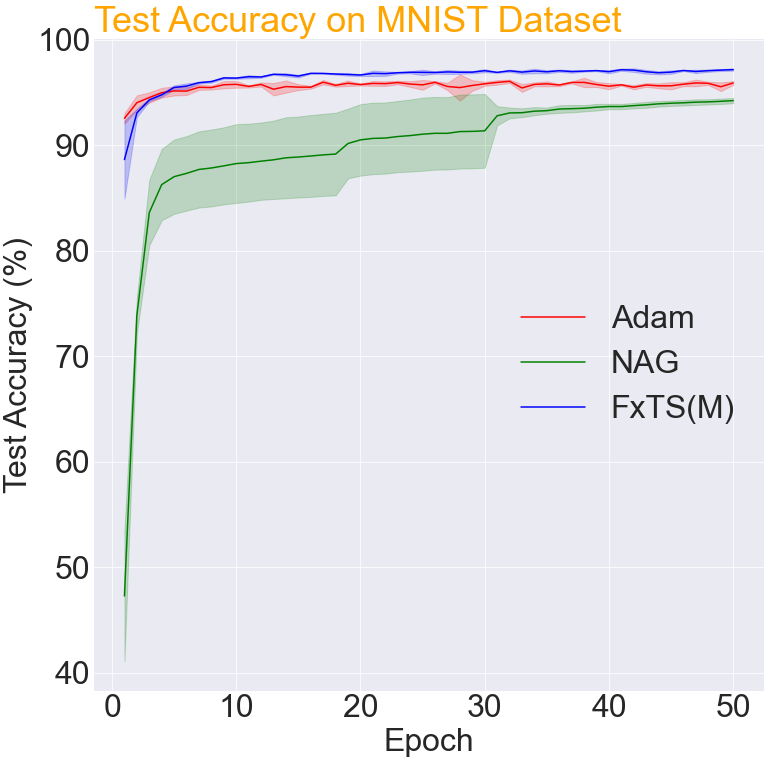}\cr
			(a)&(b)&(c)\cr
			\includegraphics[width=0.5\columnwidth]{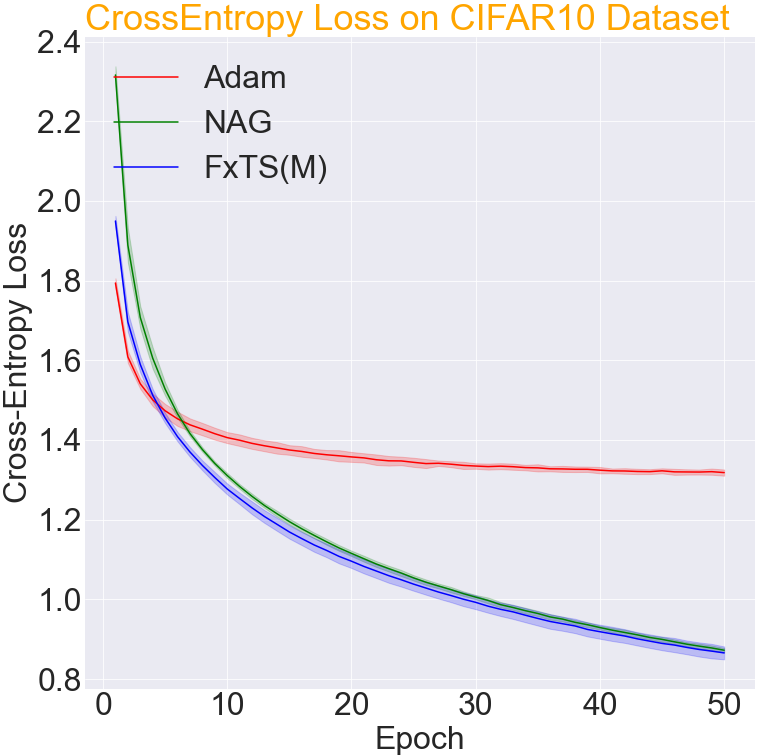}&
			\includegraphics[width=0.5\columnwidth]{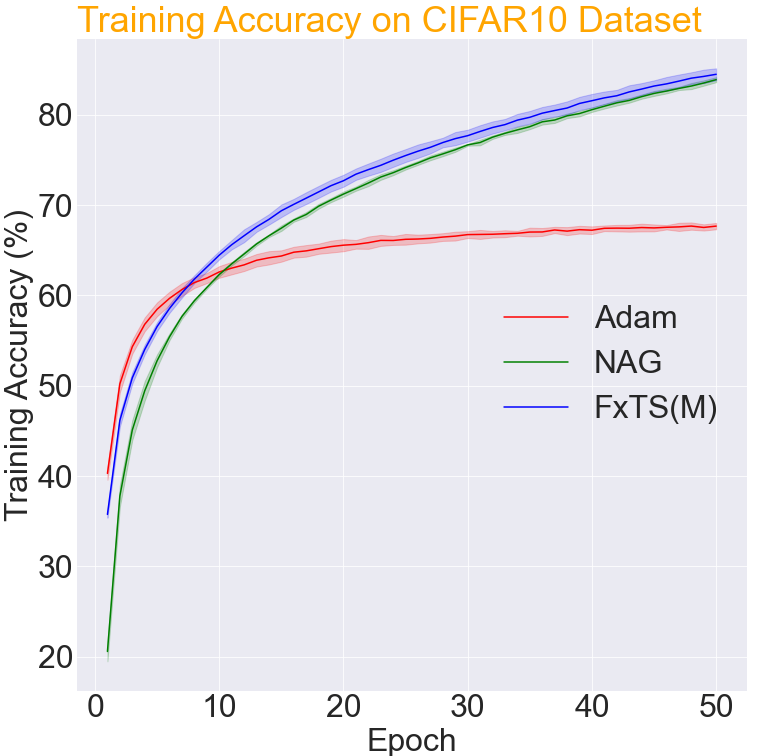}&
			\includegraphics[width=0.5\columnwidth]{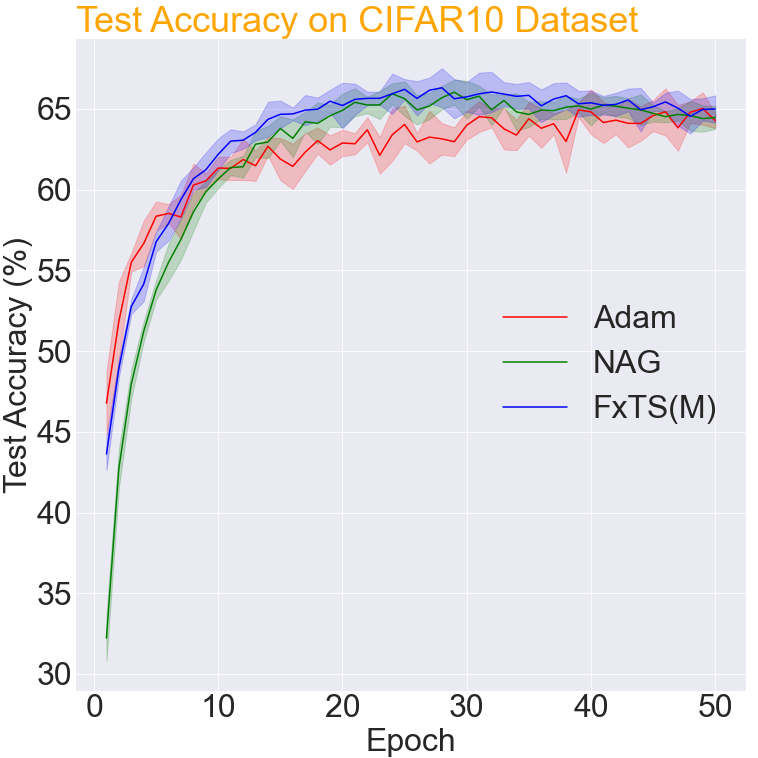}\cr
			(d)&(e)&(f)
		\end{tabular}
		\caption{Comparison of several optimization algorithms for training deep neural networks on MNIST and CIFAR datasets across five random seeds. FxTS(M) outperforms Adam and NAG optimizers on various performance measures.}
		\label{fig:NN}
	\end{center}
	\vspace{-1.5em}
\end{figure*}

\vspace{-.5em}\section{Experiments} \label{sec: experiments}
In this section, we present empirical results on optimizing (non-convex) functions that satisfy PL-inequality and training deep neural networks. The algorithms were implemented using PyTorch~0.4.1 on a 16GB Core-i7 2.8GHz CPU and NVIDIA GeForce GTX-1060 GPU.\\
\noindent {\bf Algorithms.} We compare the proposed FxTS-GF and FxTS(M)-GF algorithms against the state-of-the-art Adam~\cite{kingma2015adam} and the Nesterov accelerated gradient (NAG) descent~\cite{polyak1964some, sutskever2013importance}. The hyperparameters for different optimizers are tuned for optimal performance. We use constant step-size for all the algorithms.

\noindent{\bf Datasets.} For the purpose of implementing deep neural networks, we examine the performances of the aforementioned algorithms on two widely used datasets: MNIST (60000 training samples, 10000 test samples)~\cite{lecun1998gradient}, and CIFAR10 (50000 training samples, 10000 test samples)~\cite{krizhevsky2009learning}.

\subsection{Optimizing Rosenbrock Function}\label{subsec:Rosenbrock}
Rosenbrock function~\cite{rosenbrock1960automatic} is a \emph{non-convex} function with a global minimum at $(1,1)$, and is often used to benchmark optimization algorithms. The global minimum resides in a long, narrow, parabolic-shaped flat valley. While it is easy to locate the valley, the convergence of optimization algorithms to a global minimum is difficult. The Rosenbrock function is given by:
\begin{equation*}
	f(x_1,x_2) = (1 - x_1)^2 + 100(x_2 - x_1^2)^2.
\end{equation*}
Despite it being a non-convex function, the Rosenbrock function can be shown to satisfy PL-inequality~\eqref{plineq} with modulus $\mu=0.1$ in the region $[-1,1]\times[-1,1]$. We evaluate the convergence behavior of various optimization algorithms for the initial point $(x_1,x_2)=(0.3,0.8)$ and constant learning rate $10^{-3}$. We use the following hyperparameters for the optimization algorithms:\\
\noindent {\bf Adam}: $\beta'{s} = (0.9, 0.999)$, $\epsilon = 10^{-8}$\\
\noindent {\bf NAG}: Momentum = 0.5\\
\noindent {\bf FxTS-GF}: $\beta'{s} = (1.25, 1.25)$, $\alpha'{s} = (20, 1.98)$\\
\noindent {\bf FxTS(M)-GF}: $\beta'{s} = (1.25, 1.25)$, $\alpha'{s} = (20, 1.98)$, Momentum = 0.18

Figure~\ref{fig:Rosenbrock}a plots the evolution of the norm of the error term $x-x^\star$ for various optimization algorithms. It can be seen that the proposed FxTS-GF and FxTS(M)-GF algorithms converge much faster than the Adam and NAG optimizers. The detailed evolution of descent trajectories of the aforementioned optimization algorithms for varying initial conditions can be found in the supplementary material. The performance of the FxTS(M)-GF algorithm for randomly chosen initial conditions is shown in Figure~\ref{fig:Rosenbrock}b on a semilog-scale. A straight line on a semilog-scale depicts exponentially fast convergence. However, the proposed FxTS(M)-GF is shown to achieve faster than exponential convergence to the global minimum, independent of initialization. Additional results can be found in the supplementary material.


\subsection{Training Deep Neural Networks}\label{subsec:TrainDNN}
The accelerated convergence behavior of the FxTS(M)-GF is further evaluated by training deep neural networks on MNIST and CIFAR10 datasets, respectively. As before, the performance of the FxTS(M)-GF algorithm is benchmarked against the Adam and NAG optimizers for three criteria: (i) minimization of training loss, (ii) training accuracy, (iii) accuracy on the test set (generalization). The neural network architecture for MNIST consists of a convolutional layer (with 32 filters of size $3\times3$), followed by a dense layer of output size 128. The final layer consists transforms the 128-dimensional input into a 10-dimensional output (corresponding to ten classes). We employ the ReLU activation function for the convolutional and the first linear layer, and the softmax activation function for the output layer. The loss function is the cross-entropy along with $l_2$-regularization (coefficient 0.01). The learning rates for Adam and NAG are kept at $10^{-3}$. A larger learning rate for Adam and NAG seems to destabilize the learning curve. On the other hand, the learning rate for the FxTS(M)-GF is chosen as $0.005$. The momentum parameters for the NAG and FxTS(M)-GF algorithms are chosen as 0.5 and 0.3, respectively. The training loss vs epoch is presented in Figure~\ref{fig:NN}a, while the training and testing accuracies are depicted in Figures~\ref{fig:NN}b and \ref{fig:NN}c, respectively. These figures depict average performances across five random seeds. As can be seen, our FxTS(M)-GF achieves the lowest training loss on the MNIST dataset. Moreover, this performance gain also translates into better performance on training and testing accuracies.

For evaluating optimizers on the CIFAR10 dataset, we consider a neural network architecture with two convolutional layers (6 filters of size $5\times5$, 16 filters of size $5\times5$), each followed by a max-pooling layer (with a $2\times 2$ window). The architecture also consists of three fully connected layers of output sizes 120, 84, and 10 (number of classes), respectively. We employ the ReLU activation function for the convolutional and the first two linear layers, and the softmax activation function for the output linear layer. The learning rates for all the optimizers are chosen as $10^{-3}$, while the momentum parameters for the NAG and the FxTS(M)-GF are chosen as 0.5. The training loss (averaged across five random seeds) vs epoch is presented in Figure~\ref{fig:NN}d, while the training and testing accuracies are depicted in Figures~\ref{fig:NN}e and \ref{fig:NN}f, respectively. As can be seen, the proposed FxTS(M)-GF achieves the lowest training loss on the CIFAR10 dataset, too, along with better performance on training and testing accuracies. Interestingly, the training curve with Adam optimizer plateaus quite early during the training.\vspace{-1em}

\section{Conclusion} \label{sec: conclusion}
In this paper, we leverage continuous-time stability theory to develop novel optimization algorithms with accelerated convergence guarantees. In particular, we demonstrate that a class of continuous-time dynamical systems, suitably designed to track the minimum of convex objective functions, can do so in a fixed time independent of initialization. The resulting continuous-time dynamics are shown to be consistent upon discretization. The continuous-time dynamical system also comprises two desirable characteristics: (a) robustness to additive perturbations, (b) constant regret bounds. As an extension to data-driven learning, we also develop a momentum-based fixed-time convergent gradient flow scheme. The equivalent discretized algorithm is validated on several examples consisting of training of neural networks and minimization of invex functions. The proposed FxTS(M)-GF scheme outperforms Adam and NAG optimizers on several performance measures.

\pagebreak

\bibliography{aaai22}
\end{document}